\documentclass[a4paper]{article}

\usepackage{amsfonts}
\usepackage{amsmath}
\usepackage{amsthm}
\usepackage{graphicx}
\usepackage{subfigure}
\usepackage{txfonts}
\usepackage{verbatim}

\newtheorem{thm}{Theorem}[section]
\newtheorem{lem}[thm]{Lemma}
\newtheorem{prop}[thm]{Proposition}
\newtheorem{cor}[thm]{Corollary}

\theoremstyle{definition}
\newtheorem{ex}[thm]{Example}
\newtheorem{rmrk}[thm]{Remark}
\newtheorem{defn}[thm]{Definition}

\newcommand{\ilim}{\mathop{\varinjlim}\limits}

\newcommand{\sha}{{\cyr Sh}}

\newcommand{\triv}{\boldsymbol{1}}
\newcommand{\bZ}{\mathbb{Z}}
\newcommand{\bQ}{\mathbb{Q}}
\newcommand{\bF}{\mathbb{F}}
\newcommand{\bC}{\mathbb{C}}
\newcommand{\fm}{\mathfrak m}
\newcommand{\fn}{\mathfrak n}
\newcommand{\fp}{\mathfrak p}
\newcommand{\ff}{\mathfrak f}
\newcommand{\fq}{\mathfrak q}
\newcommand{\frakS}{\mathfrak S}
\newcommand{\modm}{(\fm_0,\;\fm_\infty)}

\newcommand{\C}{{\mathcal C}}

\newcommand{\pSelmer}{S_{\!\!p}}
\newcommand{\pInftySelmer}{S_{\!\!p^\infty}}
\newcommand{\ord}{{\rm ord}}
\newcommand{\Ind}{{\rm Ind}}
\newcommand{\Res}{{\rm Res}}
\newcommand{\Gal}{{\rm Gal}}

\long\def\symbolfootnote[#1]#2{\begingroup
\def\thefootnote{\fnsymbol{footnote}}\footnote[#1]{#2}\endgroup}

\input cyracc.def
\newfam\cyrfam
\font\tencyr=wncyr10
\font\sevencyr=wncyr7
\font\fivecyr=wncyr5
\def\cyr{\fam\cyrfam\tencyr\cyracc}
\textfont\cyrfam=\tencyr \scriptfont\cyrfam=\sevencyr
\scriptscriptfont\cyrfam=\fivecyr

\font\teneurm=eurm10
\font\seveneurm=eurm7
\font\fiveeurm=eurm5
\newfam\eurmfam
\textfont\eurmfam=\teneurm
\scriptfont\eurmfam=\seveneurm
\scriptscriptfont\eurmfam=\fiveeurm
\def\eurm#1{{\fam\eurmfam\relax#1}}

\font\teneusm=eusm10
\font\seveneusm=eusm7
\font\fiveeusm=eusm5
\newfam\eusmfam
\textfont\eusmfam=\teneusm
\scriptfont\eusmfam=\seveneusm
\scriptscriptfont\eusmfam=\fiveeusm
\def\eusm#1{{\fam\eusmfam\relax#1}}
\begin{document}
\title{Large Selmer groups over number fields}
\author{Alex Bartel$^\dagger$}

\maketitle

\begin{abstract}
Let $p$ be a prime number and $M$ a quadratic number field, $M\neq\bQ(\sqrt{p})$ if $p\equiv
1 {\rm\;mod\;}4$. We will prove that for any positive integer $d$ there exists a Galois extension
$F/\bQ$ with Galois group $D_{2p}$ and an elliptic curve $E/\bQ$ such that $F$ contains $M$ and the
$p$-Selmer group of $E/F$ has size at least $p^d$.
\end{abstract}
\symbolfootnote[0]{$^\dagger$ The author was supported by an EPSRC grant}
\symbolfootnote[0]{MSC 2000: Primary 11G05, Secondary 11G40, 20C10}
\tableofcontents

\section{Introduction}
Let $E$ be an elliptic curve defined over a number field $K$ and let $F$ be a number field containing $K$. We will write $r(E/F)$ for the rank of the abelian group of $F$-rational points of $E$, $S_{\!\!p}(E/F)$ for the $p$-Selmer group of $E$ over $F$, $\pInftySelmer(E/F)$ for the $p$-infinity Selmer group of $E$ over $F$ and $\sha(E/F)$ for the Tate-Shafarevich group of $E$ over $F$. It is conjectured that $\sha(E/F)$ is always finite and in this case the $p$-infinity Selmer group contains a direct product of $r(E/F)$ copies of $\mathbb{Q}_p/\bZ_p$. On the other hand, it is widely believed that both $r(E/F)$ and $\sha(E/F)[p^\infty]$ can be arbitrarily large for any $p$ when $F=K=\mathbb{Q}$.

Currently it is not even known whether the rank $r(E/F)$ can be arbitrarily large as $K$ and $F$ vary over number fields of bounded degree over $\mathbb{Q}$. The corresponding statement for $\sha(E/F)[p^\infty]$ was proved in \cite{Klo-1}. The bound given there for the degree of $F/\mathbb{Q} = K/\mathbb{Q}$ is ${\eurm O}(p^4)$. For elliptic curves with complex multiplication this is improved in \cite{Cl-1} to $K=\bQ$ and $\left[F:\bQ\right] < p^3$. Moreover, this latter construction does not need to vary the elliptic curve. For $p=2,3,5$ it is shown in \cite{Kra-1}, in \cite{Cas-1} and in \cite{Fis-1}, respectively, that $\sha(E/\mathbb{Q})[p^\infty]$ can be arbitrarily large and for $p=7,13$ the same is shown in \cite{Mat-1}.

In general, very little is known about the Tate-Shafarevich group. A more tractable question is whether the $p$-Selmer group $\pSelmer(E/F)$ can be arbitrarily large as $K$ and $F$ vary over number fields of bounded degree and $E$ is defined over $K$. Since $\pInftySelmer(E/F)$ contains $ \ilim E(F)/p^nE(F)$ (where the direct limit is taken with respect to multiplication-by-$p$ maps) with quotient isomorphic to $\sha(E/F)[p^\infty]$, an affirmative answer to this question implies that either $r(E/F)$ or $\sha(E/F)[p]$ must get arbitrarily large. Tom Fisher showed in \cite{Fis-1} that $S_{\!7}(E/\mathbb{Q})$ can get arbitrarily large. Unlike for $p=5$, the result does not specify whether the large Selmer group comes from a large rank or from a large Tate-Shafarevich group. In \cite{KS-1} it is shown that the $p$-Selmer group $\pSelmer(E/K)$ can be arbitrarily large when $F=K$ varies over number fields of degree at most $g+1$ where $g$ is the genus of the modular curve $X_0(p)$ and in particular that $S_{\!13}(E/\mathbb{Q})$ can be arbitrarily large. The main result of the present paper is the following:

\begin{thm}\label{thm:Res}
Let $p$ be a prime number and $M$ a quadratic number field, $M\neq\bQ(\sqrt{p})$ if $p\equiv
1 {\rm\;mod\;}4$. Given any positive integer $d$ there exists a Galois extension
$F/\bQ$ with Galois group $D_{2p}$ and an elliptic curve $E/\bQ$ such that $F$ contains $M$ and $\#\pSelmer(E/F)\geq p^d$.
\end{thm}

Our approach lends itself to concrete computations. We will write down the elliptic curves $E$ over $\mathbb{Q}$ in Legendre normal form and exhibit the number fields $F$ in terms of class field theory. Defining polynomials can then be computed using algorithms described in \cite[Chapter 6]{Coh-1}.
\begin{rmrk}
In a recent preprint \cite{Mat-2}, Matsuno shows that given any cyclic extension $F/\bQ$ of prime degree $p$, the $p$-torsion of the Tate-Shafarevich group of $E/F$ is unbounded as $E$ varies over elliptic curves defined over $\bQ$.
Also, in a recent preprint by Clark and Sharif \cite{CF-1} it is shown that for any $E/\bQ$ the Tate-Shafarevich group can get arbitrarily large over extensions $F/\bQ$ of degree $p$, not necessarily Galois.
\end{rmrk}
\begin{rmrk}
Unfortunately, it seems impossible to push our construction further to lower the degree of the extensions required. The technique fundamentally relies on non-abelian extensions of degree divisible by $p$.
\end{rmrk}
We now briefly describe the method of our proof. The starting point is the conjecture of Birch, Swinnerton-Dyer and Tate \cite{Tat-1} relating the order of vanishing of the L-function $L(E/F,s)$ over $F$, associated to the elliptic curve $E$, at $s=1$ to several arithmetic invariants of $E/F$. We will call the corresponding quotient the BSD-quotient for $E/F$.

If for some fields $L_i$ and $L_j'$ we have an identity of L-functions
\[\prod_i L(E/L_i,s) = \prod_j L(E/L_j',s)\]
then the conjecture predicts an equality of the corresponding products of BSD-quotients. In fact, this equality follows just from the finiteness of the Tate-Shafarevich groups, as explained in the footnote on page 7 in \cite{TVD-2}. We use such an equality to relate certain quotients of orders of Tate-Shafarevich groups and of regulators to a corresponding quotient of Tamagawa numbers. In particular, we can show that if this quotient of Tamagawa numbers gets arbitrarily large then the quotient
\[\prod_i\#\sha(E/L_i)\times \text{Reg}(E/L_i) \Big/ \prod_j\#\sha(E/L_j')\times \text{Reg}(E/L_j')\]
must get arbitrarily large, where $\text{Reg}(E/F)$ is the absolute value of the determinant of the canonical height pairing on a basis of $E(F)/E(F)_{\text{tors}}$.

We then examine so called regulator constants, first introduced in \cite{TVD-1} and further exploited in \cite{TVD-2}, to show that an arbitrarily large regulator quotient implies arbitrarily large rank.

To produce arbitrarily large Tamagawa quotients is an application of class field theory and is done in section 3.

In section 4 we prove Theorem \ref{thm:Res} unconditional on the finiteness of the Tate-Shafarevich groups. In this case the BSD quotient must be modified to account for a possible divisible component in the Tate-Shafarevich groups. We present the conditional proof first because it appears more natural in the context of the conjecture of Birch, Swinnerton-Dyer and Tate and motivates the unconditional one.\vspace{\baselineskip}
\newline
\textbf{Acknowledgements.} During this research the author was supported by an EPSRC grant at Cambridge University. We would like to thank the Department of Pure Mathematics and Mathematical Statistics for a wonderful working environment. We would also like to thank Antonio Lei and Vladimir Dokchitser for many useful discussions. This work would not have been possible without the invaluable patient and constructive guidance from Tim Dokchitser and without Tim's and Vladimir's enormous time commitment! Finally, thanks are due to an anonymous referee for an extremely careful reading of the manuscript and lots of constructive criticism!\vspace{\baselineskip}
\newline
\textbf{Notation.} Throughout the paper $K$ will be a number field, $\bar{K}$ will denote an algebraic closure. If $v$ is a place of $K$ then $|.|_v$ will denote the normalised absolute value at $v$. The absolute Galois group $\text{Gal}(\bar{K}/K)$ of $K$ will be denoted $G_K$. Given an elliptic curve $E/K$ we use the following notation:

\begin{tabular}{ l l }
  $\text{r}(E/K)$ & the Mordell-Weil rank of $E/K$; \\
	$c_v(E/K)$ & the local Tamagawa number at a place $v$ of $K$; \\
	$c_v(E/F)$ & the product of the local Tamagawa numbers at all places of $F$ above\\
	 & $v$ where $F/K$ is an extension of number fields and $v$ is a place of $K$; \\
	$c(E/K)$ & the product of the local Tamagawa numbers at all finite places of $K$; \\
	$W_{F/K}(E)$ & the Weil restriction of scalars of $E$ from $F$ to $K$; \\
	$\pSelmer(E/F)$ & the $p$-Selmer group of $E/F$, defined as\\
	 & $\text{ker}\left(H^1(G_F,E[p])\rightarrow \prod_v H^1(G_v,E)\right)$, where $G_v = \text{Gal}(\bar{F_v}/F_v)$, the\\
	 & map is the restriction and the product is taken over all places of $F$.
\end{tabular}

Fix an invariant differential $\omega$ on $E$. At each finite place $v$ of $K$ take a N\'{e}ron differential $\omega_v^0$. Then we set $C_v(E/K) = c_v(E/K)\left|\frac{\omega}{\omega_v^0}\right|_v$, $C_v(E/F) = 
\prod_{w|v} C_w(E/F)$ with the product taken over places of $F$, and
\[C(E/K) = \prod_{v\nmid\infty}C_v(E/K).\]
Here we followed \cite{Tat-1} in writing $\frac{\omega}{\omega_v^0}$ for the unique $v$-adic number $\delta$ such that $\omega=\delta\omega_v^0$, which exists because the space of holomorphic differentials on an elliptic curve is one-dimensional.

The definition of $C(E/K)$ depends on the choice of the invariant differential $\omega$ but this dependence will not cause any ambiguity as long as we always choose the same differential when we have the same expression for number fields $L_i/K$.

\section{Regulator constants}
\subsection{BSD quotients}
We recall a compatibility statement between the Birch and Swinnerton-Dyer conjecture and Artin formalism for $L$-functions and refer to \cite{TVD-1} for details. If $G$ is a group, $H$ a subgroup and $\rho$ a representation of $H$, we write $\text{Ind}_{G/H}(\rho)$ for the induced representation from $H$ to $G$. Similarly, given a representation $\tau$ of $G$, we write $\text{Res}_{G/H}(\tau)$ for the restriction of $\tau$ to $H$. If $K$ is a number field and $\bar{K}$ is a fixed algebraic closure and if $\bar{K}\supseteq L\supseteq K$ is an algebraic extension of number fields then given an Artin representation $\rho$ of $\text{Gal}(\bar{K}/L)$ we write $\text{Ind}_{L/K}(\rho)$ for the induced representation of $\text{Gal}(\bar{K}/K)$.

From now on, whenever we start with a number field, we fix an algebraic closure and we assume all extensions to be contained in this algebraic closure. Given a number field $K$, suppose that we have algebraic extensions $L_i\supseteq K$ and $L_j'\supseteq K$ such that
\[\bigoplus_i \text{Ind}_{L_i/K} \boldsymbol{1}_{L_i} \cong \bigoplus_j \text{Ind}_{L_j'/K} \boldsymbol{1}_{L_j'}.\]
Then by Artin formalism for $L$-functions we have
\[\prod_i L(E/L_i,s) = \prod_j L(E/L_j',s).\]
Of great importance to us is the following compatibility statement between Artin formalism and the Birch and Swinnerton-Dyer conjecture:
\begin{thm}[\cite{TVD-1}, Theorem 2.3]\label{thm:compThm}
If $L_i$ and $L_j'$ are as above then
\begin{enumerate}
\item $\sum_i r(E/L_i) = \sum_j r(E/L_j')$;
\item assuming that $\sha(E/L_i)$ and $\sha(E/L_j')$ are finite we have
\[\prod_i \frac{\#\sha(E/L_i){\rm Reg}(E/L_i)C(E/L_i)}{|E(L_i)_{\rm tors}|^2} = \prod_j \frac{\#\sha(E/L_j'){\rm Reg}(E/L_j')C(E/L_j')}{|E(L_j')_{\rm tors}|^2}.\]
\end{enumerate}
Moreover, if one only assumes that the $p$-primary parts of the Tate-Shafarevich groups are finite then the same equality holds up to a $p$-adic unit if $\sha$ is replaced by its $p$-primary part.
\end{thm}

Note that the real and the complex periods as well as the discriminants of the fields, which are present in the conjecture of Birch and Swinnerton-Dyer, cancel in our situation, provided that one chooses the same invariant differential $\omega$ over $K$ for each term.

Now let $G$ be any finite group and write ${\eusm H}$ for the set of conjugacy classes of subgroups of $G$. If for some $H_i\in {\eusm H}$ and $H'_j\in {\eusm H}$ we have an isomorphism of permutation representations of $G$ $\oplus_i \mathbb{C}[G/H_i] \cong \oplus_j \mathbb{C}[G/H'_j]$ then we follow \cite{TVD-1} in saying that
\[\Theta = \sum_i H_i - \sum_j H'_j \in \bZ{\eusm H}\]
is a relation between permutation representations, or just a $G$-relation.

We can reformulate this definition as follows: the Burnside ring of $G$ is defined as the ring of formal $\bZ$-linear combinations of isomorphism classes $[S]$ of finite $G$-sets modulo the relations
\[[S] + [T] = [S\sqcup T],\;\;\;\;[S][T] = [S\times T],\]
where $S\sqcup T$ denotes disjoint union and $S\times T$ denotes the Cartesian product; the representation ring of $G$ over $\bC$ is the ring of formal $\bZ$-linear combinations of isomorphism classes $[M]$ of finite dimensional $\bC G$-modules modulo the relations
\[[M] + [N] = [M\oplus N],\;\;\;\;[M][N] = [M\otimes N].\] We have a natural map from the Burnside ring to the representation ring that sends a $G$-set $S$ to the $\bC G$-module $\bC\left[S\right]$ with $\bC$-basis indexed by the elements of $S$ and the natural $G$-action. For more material on the Burnside ring and the representation ring, see for example \cite[\S 80-81]{CR}. A $G$-relation is then an element of the kernel of this natural map.

\textbf{Notation.} Let $F/K$ be a Galois extension of number fields with Galois group $G$. Given a $G$-relation $\Theta$ as above, set $L_i = F^{H_i}$ and $L_j' = F^{H'_j}$. Write $\text{Reg}(E/\Theta)$ for the corresponding quotient $\prod_i \text{Reg}(E/L_i)\big/\prod_j \text{Reg}(E/L_j')$ and similarly for $\#\sha(E/\Theta)$, $C(E/\Theta)$ and $|E(\Theta)_{\text{tors}}|$ or indeed for any function to $\bC$ associated with $E$ which depends on the field extension.

In this shorthand language the second part of Theorem \ref{thm:compThm} says that
\begin{eqnarray}\label{eq:mainEqn}\#\sha(E/\Theta)\text{Reg}(E/\Theta) = \frac{|E(\Theta)_{\text{tors}}|^2}{C(E/\Theta)}.
\end{eqnarray}
Our strategy will be to control the Tamagawa numbers and the torsion subgroups to make the left hand side of this equation large.
%Main Example
\begin{ex}\label{ex:impEx}
Let $G=\left<a,b: a^p = b^2 = (ab)^2=1\right>$ be the dihedral group of order $2p$ for an odd prime $p$. Then we have a relation between permutation representations
\[\Theta = 1 - 2C_2 - C_p + 2G\]
which is unique up to scalar multiples. Suppose now that $E/K$ is an elliptic curve. Take the subgroups $H=\left<a\right>=C_p$, $H'=\left<b\right>=C_2$. Here and in the rest of the paper we will often identify subgroups with their image in ${\eusm H}$. Let $F/K$ be a Galois extension of number fields with Galois group $G$ and let $L=F^{H'}$, $M=F^H$ be intermediate extensions. Let $v$ be a finite place of $K$. If $E$ has split multiplicative reduction at $v$ then for any extension $K'/K$ and any place $w$ of $K'$ above $v$ we have $c_w(E/K') = -w(j(E))$ where $j(E)$ is the $j$-invariant of the elliptic curve (see e.g. \cite[Ch. IV Cor. 9.2]{Sil-2}). Thus, if $v$ is a place of split multiplicative reduction of $E$ with only one prime of $F$ above $v$ and this prime has ramification index $p$ then
\[c_v(E/\Theta) = \frac{c_v(E/K)^2c_v(E/F)}{c_v(E/M)c_v(E/L)^2} = \frac{pc_v(E/K)^3}{p^2c_v(E/K)^3} = \frac{1}{p}.\]
Similarly, it is easily seen that if a place $v$ of split multiplicative reduction is totally ramified in $F/K$ then the associated Tamagawa quotient is $1/p$ and for any other ramification behaviour it is 1.
\end{ex}
\begin{rmrk}
Given any relation $\Theta$, if $E$ is semi-stable then $C(E/\Theta) = c(E/\Theta)$. Indeed, we first claim that in a relation the Tamagawa quotient
\[\prod_i C_v(E/L_i)\Big/\prod_j C_v(E/L_j')\]
above each finite place $v$ of $K$ does not depend on the choice of the invariant differential $\omega$. Namely,
if we re-scale $\omega$ by a constant $\alpha\in K^\times$ then the Tamagawa quotient changes by a factor of
\begin{eqnarray*}
\prod_i \prod_{\substack{w\text{ places }\\ \text{of }L_i,\;w|v}} |\alpha|_w\Big/
\prod_j \prod_{\substack{w\text{ places }\\ \text{of }L_j',\;w|v}} |\alpha|_w & = &
\prod_i |\alpha|_v^{\#\{\text{places of }L_i\text{ above v}\}}
\Big/ \prod_j |\alpha|_v^{\#\{\text{places of }L_j'\text{ above v}\}}\\
& = & |\alpha|_v^{\sum_i \#D\backslash G/H_i - \sum_j \#D\backslash G/H_j'},
\end{eqnarray*}
where $D$ is a decomposition group of $v$.
But the exponent of $|\alpha|_v$ is zero by Mackey's formula and Frobenius reciprocity, since $\#D\backslash G/H$
is the number of trivial representations in the direct sum decomposition of $\Res_{G/D}(\Ind_{G/H}(\triv_H))$, $\triv_H$ denoting
the trivial representation of $H$, and since we have
an isomorphism of permutation representations $\oplus_i \Ind_{G/H_i}(\triv_{H_i})\cong \oplus_j
\Ind_{G/H_j'}(\triv_{H_j'})$.

Next, we note that when $v$ is a place of semi-stable reduction of $E$ we can choose $\omega$ to be a N\'eron differential at $v$. Then $\omega$ stays minimal at all places above $v$ and so in a relation we can replace $C_v$ by $c_v$ in this case. Thus, for semi-stable elliptic curves $E$ we can replace the product $C(E/K)$ of the modified Tamagawa numbers by just the product of the local Tamagawa numbers $c(E/K)$, as claimed.
\end{rmrk}
\subsection{Regulator constants and Mordell-Weil ranks}
Let $G$ be a finite group, let $\Gamma$ be a finitely generated $\bZ G$-module which is $\bZ$-free. We call such a module a $G$-lattice over $\bZ$. Let $\left<,\right>$ be a fixed $\bC$-valued $G$-invariant non-degenerate $\bZ$-bilinear pairing on $\Gamma$. For each subgroup $H\leq G$,
\[\Gamma^H = \left\{\gamma\in\Gamma:\;\gamma^h = \gamma\;\;\forall h\in H\right\}\]
is also $\bZ$-free.
\begin{defn}
Given a $G$-relation $\Theta = \sum_i H_i - \sum_j H'_j$ define the regulator constant
\[{\C}_\Theta(\Gamma) = \frac{\prod_i \text{det}\left(\frac{1}{|H_i|}\left<,\right>|\Gamma^{H_i}\right)}{\prod_j \text{det}\left(\frac{1}{|H'_j|}\left<,\right>|\Gamma^{H'_j}\right)} \in \bC^\times,\]
where each inner product matrix is evaluated with respect to some $\bZ$-basis on the fixed submodule. If the matrix of the pairing on $\Gamma^{H}$ with respect to some fixed basis is $M$ then changing the basis by the change of basis matrix $X\in \text{GL}(\Gamma^{H})$ changes the matrix of the pairing to $X^{\rm tr}MX$. So the regulator constant is indeed a well-defined element of $\bC^\times/(\bZ^\times)^2=\bC^\times$.
\end{defn}

We recall the relevant results from \cite{TVD-2}:
\begin{thm}[\cite{TVD-2}, Theorem 2.17]
The value of ${\C}_\Theta(\Gamma)$ is independent of the choice of the pairing.
\end{thm}
In particular, we can choose $\left<,\right>$ to be $\bQ$-valued and so ${\C}_\Theta(\Gamma)$ is in fact an element of $\bQ^\times$.
\begin{cor}[\cite{TVD-2}, Corollary 2.18]
${\C}_\Theta(\Gamma)$ is multiplicative in $\Theta$ and in $\Gamma$, i.e.
\begin{eqnarray*}{\C}_\Theta(\Gamma\oplus\Gamma') = {\C}_\Theta(\Gamma){\C}_\Theta(\Gamma'),\\
{\C}_{\Theta+\Theta'}(\Gamma) = {\C}_\Theta(\Gamma){\C}_{\Theta'}(\Gamma).
\end{eqnarray*}
\end{cor}
The following result explains the fundamental r\^ole of regulator constants in our construction:
\begin{prop}\label{prop:largeRank}
Let $G$ be a finite group, $\Theta$ a $G$-relation, $K$ a number field. Suppose that there exists a sequence of elliptic curves $E_i/K$ and of Galois extensions $F_i/K$ with Galois group $G$ such that for some rational prime $p$, $|{\rm ord}_p({\rm Reg}(E_i/\Theta))|\rightarrow\infty$ as $i\rightarrow\infty$. Then $r(E_i/F_i)\rightarrow\infty$ as $i\rightarrow\infty$.
\end{prop}
\begin{proof}
$E_i(F_i)/E_i(F_i)_\text{tors}$ is a $\bZ$-free $\bZ G$-module and the regulator is defined as the determinant of the N\'eron-Tate height pairing on a $\bZ$-basis of $E_i(F_i)/E_i(F_i)_\text{tors}$ so
\[\text{Reg}(E_i/\Theta)={\C}_\Theta(E_i(F_i)/E_i(F_i)_\text{tors}),\]
where $\text{Reg}(E_i/\Theta)$ is the quotient of regulators of $E$ over subfields of $F_i$. Now, by the Jordan-Zassenhaus theorem on integral representations \cite{Zas-1}, there exist only finitely many isomorphism classes of $\bZ$-free $\bZ G$-modules $\Gamma$ such that $\Gamma\otimes\bQ$ is isomorphic to a given finite dimensional rational representation. It follows that there are only finitely many such modules up to isomorphism of given rank over $\bZ$. Therefore, if $r(E_i/F_i)$ was bounded as $i\rightarrow\infty$, the set $\left\{ {\rm Reg}(E_i/\Theta)\right\}$ would be finite, contradicting the hypothesis of the theorem.
\end{proof}
We record another feature of regulator constants which will allow us to control the growth of the regulator quotient:
\begin{lem}\label{lem:unramExt}
Let $G$ be a finite group and let $\Theta = \sum_i H_i - \sum_j H'_j$ be a relation of permutation representations. Let $E/K$ be an elliptic curve over a number field and let $F/K$ be a Galois extension with Galois group $G$. If $v$ is a place of $K$ which is unramified in $F/K$ (or more generally for which all decomposition groups are cyclic) then $C_v(E/\Theta)=1$.
\end{lem}
\begin{proof}
Quite generally, if $D<G$ is a subgroup and $\psi$ is a function on the Burnside ring of $G$ (such as $C_v: H\mapsto C_v(E/F^H)$ for example) which can be written as
\[\psi(H) = \prod_{x\in H\backslash G/D} \psi_D(x^{-1}Hx\cap D)\]
for $\psi_D$ a function on the Burnside ring of $D$ (i.e. if $\psi$ is "D-local" in the language of \cite{TVD-2}) and if $\psi_D$ is trivial on all $D$-relations then $\psi$ is trivial on all $G$-relations. This follows from Mackey decomposition and a rather intricate formalism introduced in \cite[2.iii]{TVD-2}. In our case, if $D$ is the decomposition group of some $w/v$ in $G$ then the function $C_v$ is $D$-local. But we assumed that $D$ was cyclic and cyclic groups have no non-trivial relations. Therefore we are done.
\end{proof}

\section{Large Selmer groups}
We are now ready to begin the proof of the main result. This will consist of two explicit constructions: of the elliptic curve and of the field extension.
\subsection{Dihedral extension of number fields via class field theory}
We will follow the notation in \cite{Coh-1} so that the construction is readily implementable on a computer using the algorithms described there.
\newline
\textbf{Notation.}
For a number field $M$ we fix the following notation:
\newline
\begin{tabular}{ l l }
  $\fm = \modm$ & a modulus of $M$, where $\fm_0$ is an integral ideal \\
  & of the field and $\fm_\infty$ is a set of real embeddings. \\
  $I_\fm$ & for a given modulus $\fm$, the multiplicative group of \\
  & fractional ideals which are co-prime to $\fm_0$. \\
  $P_\fm$ & for a given modulus $\fm$, the subgroup of $I_\fm$ generated by all \\
  & principal ideals $(a)$, $a\in M^\times$, such that $a\equiv 1\text{ mod }^*\fm$ \\
  & by which we mean that $\text{ord}_\fp(a-1)\geq\text{ord}_\fp(\fm_0)$ for all $\fp$ above $\fm_0$ \\
  & and $\sigma(a)>0$ for all embeddings $\sigma\in\fm_\infty$.\\
  $(\fm,U)$ & a congruence subgroup, i.e. $\fm$ is a modulus and $P_\fm \leq U \leq I_\fm$.
\end{tabular}
\begin{defn}
Two congruence subgroups $(\fm,U_\fm)$ and $(\fn,U_\fn)$ are said to be equivalent if $I_\fm\cap U_\fn = I_\fn\cap U_\fm$. The smallest $\fn$ such that $(\fm,U_\fm)$ is equivalent to $(\fn,U_\fm P_\fn)$ is called the conductor associated to $(\fm,U_\fm)$. This is equivalent to saying that the conductor is the smallest modulus $\fn$ such that the natural map $I_\fm/U_\fm\rightarrow I_\fn/U_\fm P_\fn$ is injective.
\end{defn}
The following is one of the main results of global class field theory:
\begin{thm}\label{thm:classField}
Given any modulus $\fm$ of $M$ and any congruence subgroup $U$, there exists a unique abelian extension $F/M$ such that
\begin{eqnarray*}
I_\fm/U & \widetilde{\longrightarrow} & {\rm Gal}(F/M)\\
\alpha & \mapsto & (\alpha,F/M)
\end{eqnarray*}
is a group isomorphism, where for a prime ideal $\fp$ of $M$ $(\fp,F/M)$ is the Frobenius automorphism at $\fp$. This isomorphism is called the Artin map. Moreover, two congruence subgroups $(\fm,U_\fm)$ and $(\fn,U_\fn)$ give the same field extension if and only if they are equivalent. We have
\begin{equation}\label{eq:Artin}
(\tau\alpha,F/M) = \tau^{-1}(\alpha,F/M)\tau\;\;\forall \tau\in {\rm Aut}(M).
\end{equation}
If $K$ is a subfield of $M$ and $M/K$ is Galois then $F/K$ is Galois if and only if $\tau(U)$ is equivalent to $U$ for all $\tau\in {\rm Gal}(M/K)$. If $\tau(\fm) = \fm$ for all $\tau\in{\rm Gal}(M/K)$ then this condition simplifies to $\tau(U) = U$ for all $\tau\in{\rm Gal}(M/K)$.

The primes that ramify in $F/M$ are precisely the ones that divide the conductor $\ff$ of $(\fm,U_\fm)$ and a prime $\fp$ is wildly ramified if and only if $\fp^2$ divides $\ff$.
\end{thm}
We will now use this result to construct dihedral extensions of $\bQ$ with a prescribed intermediate field and arbitrarily many ramified primes:
\begin{thm}\label{thm:dihedralExt}
Let $M=\bQ(\sqrt{d})$ be a quadratic number field with $d$ a square-free integer, and $p$ any odd prime number. Define the following sets of rational primes:
\begin{eqnarray*}
\frakS_1 & := &\left\{q {\rm\;rational\;odd\;prime }: q\;{\rm\;splits\;in\;}M/\bQ,\;q\equiv 1{\rm\;mod\;}p\right\}\\
\frakS_2 & := &\left\{q {\rm\;rational\;odd\;prime }: q\;{\rm\;is\;inert\;in\;}M/\bQ,\;q\equiv -1{\rm\;mod\;}p\right\}.
\end{eqnarray*}
Given any positive integers $k_1$ and $k_2$ there exists a Galois extension $F/\bQ$ with Galois group $D_{2p}$ such that $F$ contains $M$ and
\begin{enumerate}
\item[(i)] no rational primes outside of $\frakS_1\cup \frakS_2$ have ramification index divisible by $p$,
\item[(ii)] at least $k_1$ primes from $\frakS_1$ ramify in $F/\bQ$ and
\item[(iii)] unless $d=p\equiv 1\;{\rm mod}\;4$, at least $k_2$ primes from $\frakS_2$ ramify in $F/\bQ$.
\end{enumerate}
The same statement holds with condition {\rm(ii)} replaced by
\begin{enumerate}
\item[(ii)'] unless $d=p\equiv 1\;{\rm mod}\;4$, no primes from $\frakS_1$ have ramification index divisible by $p$;
\end{enumerate}
\end{thm}
It is condition {\rm(ii)'} that we will need for our main result but we include the variant with condition
{\rm(ii)} for completeness.
\begin{proof}
We will find infinitely many dihedral extensions $F_i$ of $\bQ$ containing $M$ with disjoint sets of ramified primes in $F_i/M$. By taking "diagonal" subfields in their compositum we will create the required extension. To construct the extensions $F_i$ we will use the above results from class field theory by constructing moduli $\fm$ which will be fixed by the Galois group of $M/\bQ$ and such that $I_\fm/P_\fm$ will have a quotient $I_\fm/U$ of order $p$ with $U$ fixed by the Galois group of $M/\bQ$ and this Galois group acting as $x\mapsto x^{-1}$ on the quotient.

Let $\mathfrak U$ be the group of units of $M$ and for a modulus $\fm=(\fm_0,\fm_\infty)$ of $M$ define
\[\mathfrak U_\fm = \{u\in\mathfrak U: u\equiv 1\text{ mod }^*\fm\}.\]
Further define
\[I_\fm' = \{a\in M^\times:\;\text{ord}_\fp(a) = 0\;\forall \fp |\fm_0\}\]
and
\[P_\fm' = \{a\in I_\fm':\; a\equiv 1 \text{ mod }^*\fm\}.\]
Then we have the exact sequence
\begin{eqnarray}\label{eq:classGroup}
0\rightarrow \mathfrak U/\mathfrak U_\fm\rightarrow I_\fm'/P_\fm'\rightarrow I_\fm/P_\fm\rightarrow Cl({\eusm O}_M)\rightarrow 0.
\end{eqnarray}
The map $I_\fm'/P_\fm'\rightarrow I_\fm/P_\fm$ simply sends an element to the ideal it generates (or rather its equivalence class).
We will concentrate on the term $I_\fm'/P_\fm'$ for now.

First, we claim that by Dirichlet's prime number theorem both sets $\frakS_1$ and $\frakS_2$ are infinite, unless $d=p\equiv 1\;{\rm mod}\;4$, in which case $\frakS_1$ is infinite and $\frakS_2$ is empty. Indeed, this is clear when $p\neq d$. If $p=d$ and $p\equiv 3\text{ mod }4$ then
\[q\text{ splits in }M \Leftrightarrow \left(\frac{p}{q}\right) = 1 \Leftrightarrow \left(\frac{q}{p}\right) = (-1)^{\frac{q-1}{2}}\]
and so again both sets are infinite since the congruence condition modulo 4 and the congruence condition modulo $p$ can be satisfied simultaneously. If $p=d$ and $p\equiv 1\text{ mod }4$ then $\left(\frac{q}{p}\right) = \left(\frac{p}{q}\right)$
and so $q\equiv \pm 1\text{ mod }p \Rightarrow \left(\frac{q}{p}\right) = \left(\frac{p}{q}\right) = 1 \Rightarrow q$ splits in $M$.

We will henceforth assume that both sets are infinite since the proof (or rather the relevant part) just carries over to the other case. Define the following sequences of distinct moduli, always taking $\fm_\infty$ to be empty and dropping the subscript from $\fm_0$ to avoid index overload:
\[\fm_i = q_iq_i',\;q_i,q_i'\in \frakS_1,\;\;\;\widetilde{\fm_j} = \widetilde{q_j}\widetilde{q_j}',\;\widetilde{q_j},\widetilde{q_j}'\in \frakS_2\]
with all $q_i,q_i',\widetilde{q_j},\widetilde{q_j}'$ distinct.
Let $\tau$ be the non-trivial element of the Galois group of $M/\bQ$. It is clear that $\tau$ fixes all the chosen moduli. We make several easy observations:
\begin{itemize}
\item By the Chinese Remainder Theorem there is an isomorphism
\begin{eqnarray*}
I_\fm'/P_\fm' & \cong & \left(I_\fm'\cap {\eusm O}_M\right)/\left(P_\fm'\cap {\eusm O}_M\right) \\
& \cong & \left({\eusm O}_M/\fm_0\right)^\times \\
& \cong & \left({\eusm O}_M/q\right)^\times \times \left({\eusm O}_M/q'\right)^\times
\end{eqnarray*}
for $\fm=\fm_i$ or $\fm=\widetilde{\fm_j}$ and $q=q_i,q'=q_i'$ or $q=\widetilde{q_j},q'=\widetilde{q_j}'$, respectively.
\item If $q\in \frakS_1$ then writing $(q)=\fq\fq'$ in $M$ we get that
\[\left({\eusm O}/q\right)^\times = \left({\eusm O}/\fq\right)^\times\times\left({\eusm O}/\fq'\right)^\times \cong \left(\bF_q\right)^\times \times \left(\bF_q\right)^\times.\]
If $\left({\eusm O}/\fq\right)^\times=\left<x\right>$ then $\left({\eusm O}/\fq'\right)^\times=\left<y\right>$ where $y=\tau(x)$. Since $q\equiv 1{\rm\;mod\;}p$ we have that $R_\fm=\left<(x^p,1),(1,y^p),(x,y)\right>$ is a subgroup of $\left({\eusm O}/q\right)^\times$ of index $p$. Moreover, $\tau(R_\fm) = R_\fm$ and $\tau((x,1)) = (1,y) \equiv (x,1)^{-1} \text{ mod } R_\fm$.
\item If $q\in \frakS_2$ then $\left({\eusm O}/q\right)^\times = \left(\bF_{q^2}\right)^\times=\left<x\right>$, say, with the action of $\tau$ being given by $\tau(x)=x^q$. Since $q\equiv -1\text{ mod }p$, $R_\fm=\left<x^p\right>$ is a subgroup of index $p$. Moreover, $\tau(R_\fm)=R_\fm$ and $\tau(x)=x^q\equiv x^{-1}\text{ mod }R_\fm$.
\item So, for $\fm=\fm_i$ or $\fm=\widetilde{\fm_j}$, $I_\fm'/P_\fm'$ contains a quotient which is isomorphic to $\bZ/p\bZ\times \bZ/p\bZ$ on which $\tau$ acts as $x\mapsto x^{-1}$. Since, for $p$ an odd prime, in a quadratic field any quotient of $\mathfrak U$ can contain at most one copy of $\bZ/p\bZ$ we deduce from the exact sequence (\ref{eq:classGroup}) that there exists a subgroup of $I_\fm/P_\fm$ which has a quotient isomorphic to $\bZ/p\bZ$ and on which $\tau$ acts as $x\mapsto x^{-1}$. The structure theorem for abelian groups now implies that $I_\fm/P_\fm$ itself has such a quotient, $I_\fm/U_\fm$, say.
\item By Theorem \ref{thm:classField} we get, for each $\fm=\fm_i$ or $\fm=\widetilde{\fm_j}$, an abelian extension $F_\fm$ of $M$ of degree $p$ with conductor dividing $\fm$. Moreover, we have chosen $R_\fm$ and thus also $U_\fm$ in such a way that the extension $F_\fm/\bQ$ is Galois and by equation (\ref{eq:Artin}) the Galois group is $D_{2p}$.
\end{itemize}
Since only finitely many of the extensions $F_\fm/M$ can be unramified, we have constructed two sequences of distinct Galois extensions $F_i=F_{\fm_i}$ and $F_j'=F_{\widetilde{\fm_j}}$ of $\bQ$ with Galois groups $D_{2p}$ with disjoint sets of primes which ramify over $M$. In one sequence these primes lie above primes from $\frakS_1$ and in the other from $\frakS_2$. These extensions are all independent over $M$. Let $\fq_i$ ramify in $F_i/M$. We will now inductively construct an extension of $M$ which is Galois over $\bQ$ with Galois group $D_{2p}$ and in which arbitrarily many primes from $\frakS_1$ ramify. The case for $\frakS_2$ is completely analogous.

Suppose we have constructed an extension $F/M$ which is Galois over $\bQ$ with Galois group $D_{2p}$ and in which the primes $\fq_1,\ldots,\fq_k$ ramify. Consider the compositum of $F$ and $F_{k+1}$. Since the two fields are disjoint over $M$, the Galois group of their compositum is $\bZ/p\bZ \times \bZ/p\bZ = \left<g\right>\times \left<h\right>$, say. Clearly, $F$ is the maximal extension of $M$ inside $F_{k+1}F$ which is unramified at $\fq_{k+1}$ and similarly $F_{k+1}$ is the maximal extension which is unramified at $\fq$ for any $\fq \in \{\fq_1,\ldots,\fq_k\}$. Thus, taking the fixed field inside $F_{k+1}F$ of $(g,h)$ we get a Galois extension of $\bQ$ with Galois group $D_{2p}$ which is ramified at all the primes $\fq_1,\ldots,\fq_{k+1}$. This inductive procedure completes our construction.
\end{proof}

\begin{rmrk}
There are algorithms for computing the ray class group of a given modulus and for computing a defining polynomial for the field associated to a congruence subgroup. They are particularly well suited in our situation since there is a specialised efficient algorithm for totally real fields and another one for complex quadratic fields. Both are described in \cite[Chapter 6]{Coh-1}.
\end{rmrk}

\subsection{Elliptic curves in Legendre normal form and main result}
The last easy ingredient we need is:
\begin{lem}\label{lem:lemRed}
Let $E$ be an elliptic curve over $\bQ$ given in Legendre normal form by
\[E: y^2 = x(x-1)(x-\lambda)\]
where $\lambda\in\bZ$ is odd. Then
\begin{itemize}
	\item $E$ has split multiplicative reduction at all odd $q | (\lambda-1)$;
	\item $E$ has multiplicative reduction at all $q | \lambda$ and it is split multiplicative if and only if $q \equiv 1\;{\rm mod}\;4$;
	\item $E$ has potentially good reduction at 2 if and only if $\lambda \nequiv 1\;{\rm mod}\;32$. Moreover, if $\lambda \equiv 17\;{\rm mod}\;32$ then $E$ has good reduction at 2.
	\item $E$ has good reduction at all other primes.
\end{itemize}
\end{lem}
\begin{proof}
We use the standard notation for the invariants $\Delta$ and $c_4$ associated to a Weierstrass equation for $E$ (see \cite[Ch. III \S 1]{Sil-1}). If $E$ is given in Legendre normal form as above then we have
\[c_4 = 16(\lambda^2-\lambda+1)\;\text{ and }\;\Delta = 16\lambda^2(\lambda-1)^2.\]
Thus the primes of bad reduction divide $\lambda$ or $\lambda-1$. Moreover for any such odd prime $q$, $c_4$ is a $q$-adic unit and so $E$ has multiplicative reduction at $q$ (\cite[Ch. VII Prop. 5.1]{Sil-1}). To determine whether it is split or non-split we use the following criterion (\cite[p. 366]{Sil-2}): let $E/K$ be given by a Weierstrass equation with the coefficients $a_1,\ldots,a_6$ and assume that it has multiplicative reduction at a prime $\fq$, the singular point being $(0,0)$. Then the reduction is split multiplicative if and only if the polynomial $T^2+a_1T-a_2$ splits over the residue field at $\fq$.

In our case, if $q|\lambda$ then the singular point of the reduction modulo $q$ is $(0,0)$ and $a_1 = 0$, $a_2 = -\lambda-1\equiv -1\text{ mod }q$. So the polynomial splits if and only if -1 is a square modulo $q$.

If $q|\lambda-1$ then perform the change of variables $x=x'+1$. The singular point again becomes $(0,0)$ and $a_1 = 0$, $a_2 = 2-\lambda\equiv 1\text{ mod }q$ and so the polynomial always splits.

Finally, $E$ has potentially good reduction at 2 if and only if the $j$-invariant is a 2-adic integer. But $\lambda$ is odd, so
\[j = c_4^3/\Delta = 16^2(\lambda^2-\lambda+1)/\lambda^2(\lambda-1)^2\]
is a 2-adic integer if and only if $\lambda-1$ is not divisible by 32. If $\lambda \equiv 17\text{ mod }32$ then it is easily seen that the substitution $x=4x'+1$, $y=8y'+4x'$ gives a Weierstrass equation which is integral with respect to 2 and with $\Delta$ a 2-adic unit.
\end{proof}
\begin{lem}\label{lem:largeTam}
Given a prime number $p>7$, a quadratic number field $M$ such that $M\neq \bQ(\sqrt{p})$ if $p\equiv 1{\rm\;mod\;}4$ and a positive integer $n$, there exists a semi-stable elliptic curve $E/\bQ$ and a Galois extension $F/\bQ$ containing $M$ with Galois group $D_{2p}$ such that
\[C(E/\Theta) = p^{-m},\]
for $m\geq n$, where $\Theta$ is the relation of permutation representations from Example \ref{ex:impEx}.
\end{lem}
\begin{proof}
Take a dihedral extension $F$ of $\bQ$ containing $M$ such that $m\geq n$ primes $q_1,\ldots,q_m$
that are inert in $M/\bQ$ ramify in $F/M$ and no other primes of $M$ ramify in $F$.
Such an $F$ exists by Theorem \ref{thm:dihedralExt}. Take
\[\lambda=16\prod_{i=1}^{m}q_i+1.\]
Then by Lemma \ref{lem:lemRed}, $E:y^2=x(x-1)(x-\lambda)$ is semi-stable and has split multiplicative reduction at all these $q_i$. All other primes are unramified in $F/M$ and thus have cyclic decomposition groups. By Lemma \ref{lem:unramExt} and by the Example \ref{ex:impEx} the Tamagawa quotients $C_{q_i}(E/\Theta)$ are $1/p$ for each of the primes $q_{1},\ldots,q_{m}$ and are 1 for all other primes, so $E$ and $F$ are as required.
\end{proof}
\begin{thm}\label{thm:mainRes}
Let $p$ be an odd prime number and $M/\bQ$ any quadratic field but if $p\equiv 1\text{ mod }4$ then assume that $M\neq\bQ(\sqrt{p})$. Assume that $p$-primary parts of Tate-Shafarevich groups of elliptic curves over number fields are always finite. Then the quantity 
\[p^{r(E/F)}\times\#\sha(E/F)[p^\infty]\]
is unbounded as $E$ varies over elliptic curves over $\bQ$ and $F/\bQ$ varies over Galois extensions with dihedral Galois group of order $2p$ containing $M$.
\end{thm}
\begin{proof}
This is known for $p\leq 7$ (see introduction for references), so for simplicity assume $p>7$. The proof will proceed by considering the $p$-part of equation (\ref{eq:mainEqn}). Let $\Theta$ be the relation of permutation representations of the dihedral group of order $2p$ considered in Example \ref{ex:impEx}. If $E/\bQ$ is a semi-stable elliptic curve then by \cite{Maz-1} and \cite{Maz-2} $E/\bQ$ has no $p$-torsion and by \cite[\S 21 Proposition 21 and remark following Lemma 6]{Ser-1} the absolute Galois group of $\bQ$ acts on $E[p]$ as $\text{GL}_2(\bF_p)$. Thus adjoining the co-ordinates of a $p$-torsion point to $\bQ$ defines an extension which is the fixed field of a Borel subgroup of $\text{GL}_2(\bF_p)$ and so has degree $p^2-1>2p$ over $\bQ$. Thus $E$ can have no $p$-torsion over $F$. Next, note that the kernel
\[\text{ker}(\pInftySelmer(E/\bQ)\rightarrow\pInftySelmer(E/F))\]
is contained in $H^1(\Gal(F/\bQ),E(F)[p^\infty])$ and is therefore 0 since $E(F)[p^\infty]$ is trivial as we have just established. But recall that the numerator of $\sha(E/\Theta)$ is $\sha(E/\bQ)^2\sha(E/F)$. Thus, in this scenario, by Proposition \ref{prop:largeRank} it suffices to show that
\[p^{\ord_p(\text{Reg}(E/\Theta))}\times\#\sha(E/\Theta)[p^{\infty}])\]
can be made arbitrarily large. So by equation (\ref{eq:mainEqn}) we need to make the $p$-part of the inverse of the Tamagawa quotient arbitrarily large. The result now immediately follows from Lemma \ref{lem:largeTam}.
\end{proof}
It remains to prove that in fact the $p$-Selmer gets large in our extensions and not just the $p^\infty$-Selmer. This will be more naturally done in the next section.
\section{Unconditional statements}\label{sec:uncond}
We now want to drop the assumption that Tate-Shafarevich groups are finite. We recall the relevant result from \cite{TVD-1}.
\begin{defn} Given an isogeny $\psi: A\rightarrow B$ of Abelian varieties over $K$, define
\begin{eqnarray*}
Q(\psi) = |\text{coker}(\psi:A(K)/A(K)_{\text{tors}}\rightarrow B(K)/B(K)_{\text{tors})}|\times\\
  \times |\text{ker}(\psi:\sha(A/K)_\text{div}\rightarrow \sha(B/K)_\text{div})|
\end{eqnarray*}
where $\sha_\text{div}$ denotes the divisible part of the Tate-Shafarevich group.
\end{defn}
\begin{thm}[\cite{TVD-1}, Theorem 4.3]
Let $\phi:A\rightarrow B$ be an isogeny of abelian varieties over a number field $K$ and $\phi^t:B^t\rightarrow A^t$ its dual isogeny. Let $\omega_A$ and $\omega_B$ be holomorphic $n$-forms on $A$ and $B$, respectively, where $n={\rm dim}\;A$ and set
\[\Omega_A = \prod_{\stackrel{v|\infty}{\text{real}}}\int_{A(K_v)}|\omega_A|.\prod_{\stackrel{v|\infty}{\text{complex}}}2\int_{A(K_v)}\omega_A\wedge\overline{\omega_A}\]
and write $\sha_0(A/K)$ for $\sha(A/K)$ modulo its divisible part, define $\Omega_B$ and $\sha_0(B/K)$ similarly. Then we have
\begin{eqnarray}\label{eq:SelmerEqn}
\frac{|B(K)_{\rm tors}|}{|A(K)_{\rm tors}|}\frac{|B^t(K)_{\rm tors}|}{|A^t(K)_{\rm tors}|}\frac{C(A/K)}{C(B/K)}\frac{\Omega_A}{\Omega_B}\prod_{q|{\rm deg}\;\phi}\frac{\#\sha_0(A/K)[q^\infty]}{\#\sha_0(B/K)[q^\infty]}=\frac{Q(\phi^t)}{Q(\phi)},
\end{eqnarray}
where the product is taken over prime divisors of ${\rm deg}\;\phi$.
\end{thm}

We will modify this equation slightly so that it resembles more closely equation (\ref{eq:mainEqn}).
Let $r$ be the common rank of the groups of $K$-rational points of the abelian varieties $A$,
$B$, $A^t$ and $B^t$ and
let $a_1,\ldots,a_r$ be elements of $A(K)$ which generate a subgroup of index $|A(K)_{\rm tors}|$ and
$a_1',\ldots,a_r'$ be elements of $A^t(K)$ which generate a subgroup of index $A^t(K)_{\rm tors}$ and let
$b_1,\ldots,b_r$ and $b_1',\ldots,b_r'$ be defined analogously for $B$.
Milne observes in \cite[I.7]{Mil-2} that the pairing which features in the Birch and Swinnerton-Dyer
conjecture is functorial, i.e.
\[\left<\phi^t(b_j'),a_i\right> = \left<b_j',\phi(a_i)\right>.\]
It follows that
\[1 = \frac{\det\left(\left<\phi^t(b_j'),a_i\right>\right)}{\det\left(\left<b_j',\phi(a_i)\right>\right)}
= \frac{\det\left(\left<a_j',a_i\right>\right)}{\det\left(\left<b_j',b_i\right>\right)}\cdot
\frac{|\text{coker}(\phi^t:B^t(K)/B^t(K)_{\text{tors}}\rightarrow A^t(K)/A^t(K)_{\text{tors}})|}
{|\text{coker}(\phi:A(K)/A(K)_{\text{tors}}\rightarrow B(K)/B(K)_{\text{tors}})|},\]
so the right hand side in equation (\ref{eq:SelmerEqn}) becomes
\[\frac{|\text{ker}(\phi^t:\sha(B^t/K)_\text{div}\rightarrow \sha(A^t/K)_\text{div})|}{|\text{ker}(\phi:\sha(A/K)_\text{div}\rightarrow \sha(B/K)_\text{div})|}\cdot \frac{\det\left(\left<b_j',b_i\right>\right)}{\det\left(\left<a_j',a_i\right>\right)}.\]
Now let $G$ be a finite group and
\[\Theta = \sum_i H_i - \sum_j H_j'\]
a $G$-relation. Let $E/K$ be an elliptic curve and $F/K$ be a Galois extension with Galois group $G$, let $L_i = F^{H_i}$ and $L_j' = F^{H_j'}$ and denote by $W_{L_i/K}(E)$, $W_{L_j'/K}(E)$ the Weil restrictions of scalars. As explained in \cite[\S 2]{Mil-1} and in \cite[\S 4]{TVD-1}, given a $G$-injection
\[f:\oplus_i\bZ[G/H_i]\rightarrow \oplus_j\bZ[G/H_j']\]
with finite co-kernel of order $d$, we can construct an isogeny of abelian varieties
\[\phi:\prod_i W_{L_i/K}(E)\rightarrow \prod_j W_{L_j'/K}(E)\]
of degree $d^2$. If we set $A=\prod_i W_{L_i/K}(E)$ and $B=\prod_j W_{L_j'/K}(E)$ then $C(A/K)/C(B/K)=C(E/\Theta)$ and $\frac{\det\left(\left<a_j',a_i\right>\right)}{\det\left(\left<b_j',b_i\right>\right)} = \C_\Theta\left(E(F)/E(F)_{\rm tors}\right)$.

We have already shown in Lemma \ref{lem:largeTam} that if we take $K=\bQ$, $G=D_{2p}$ and $\Theta$ to be the relation
from Example \ref{ex:impEx} then we can choose $E$ and $F$ such that the Tamagawa-quotient gets arbitrarily large and such that $F$ contains a predetermined quadratic subfield (subject to the restriction in the Lemma). Also, if we take $p>7$ and $E/\bQ$ semi-stable, as in Lemma \ref{lem:largeTam} then the $p$-part of
\[\frac{|B(K)_{\rm tors}|}{|A(K)_{\rm tors}|}\frac{|B^t(K)_{\rm tors}|}{|A^t(K)_{\rm tors}|}\]
is trivial by \cite{Maz-1} and \cite{Maz-2}. Finally, the real and complex periods cancel as in Theorem \ref{thm:compThm} since they are equal to the corresponding periods of the elliptic curve as explained in \cite{Mil-1}. Equation (\ref{eq:SelmerEqn}) implies that then at least one of the three quantities $\frac{\#\sha_0(A/K)[p^\infty]}{\#\sha_0(B/K)[p^\infty]}$ (and in particular $|\text{ker}(\phi:\sha_0(A/K)[p^\infty]\rightarrow \sha_0(B/K)[p^\infty])|$), $|\text{ker}(\phi:\sha(A/K)_\text{div}\rightarrow \sha(B/K)_\text{div})[p^\infty]|$ or $r(E/F)$ must get large. But for an isogeny like $\phi$ above of degree $d^2$ there exists an isogeny in the opposite direction such that their composition is multiplication by $d^2$ and thus induces the multiplication-by-$d^2$ map on the Tate-Shafarevich group. Thus $\phi$ can kill at most $p^{2\text{ord}_p(d)}$ elements of the Tate-Shafarevich group for each cyclic component and for each copy of $\bQ_p/\bZ_p$ in $\sha(A/K)_\text{div}$. It follows immediately that either the rank or the number of $p$-primary cyclic components in $\sha$ or the number of copies of $\bQ_p/\bZ_p$ in $\sha$, and in either case the $p$-Selmer group gets arbitrarily large when the Tamagawa quotient does. We therefore deduce
\begin{thm}
Given a prime number $p>7$, any non-negative integer $n$ and a quadratic field $M$ (if $p\equiv 1\text{ mod }4$ then assume $M\neq \bQ(\sqrt{p})$), there exists a semi-stable elliptic curve $E/\bQ$ and infinitely many cyclic extensions $F/M$ of degree $p$ which are Galois over $\bQ$ such that $\pSelmer(E/F)\geq p^n$.
\end{thm}

\addcontentsline{toc}{section}{References}
\bibliographystyle{plain}
\bibliography{SelmerGroups}

\textsc{St. John's College, Cambridge CB2 1TP, United Kingdom}

E-mail address: \texttt{a.bartel@dpmms.cam.ac.uk}
\end{document}